\documentclass[11pt]{amsart}

\usepackage{amsmath,amsthm,amsfonts,amssymb,bm,graphicx }

\makeatletter

\usepackage
{hyperref}
\hypersetup{colorlinks=true,citecolor=blue,linkcolor=blue,urlcolor=blue}

\theoremstyle{plain}
\newtheorem{theorem}{Theorem}
\newtheorem{lemma}[theorem]{Lemma}

\newtheorem{cor}[theorem]{Corollary}
\newtheorem{prop}[theorem]{Proposition}
\theoremstyle{definition}
\newtheorem{definition}[theorem]{Definition}

\numberwithin{equation}{section}

\newcommand{\pr}{^\prime}

\newcommand{\X}{\mathbb{X}_D}
\newcommand{\ex}{\mathbb{E}}
\newcommand{\D}{\mathcal{D}_D(x)}
\newcommand{\F}{\mathcal{F}_D(h)}
\newcommand{\sumst}{\sideset{}{^\star}\sum}

\makeatletter
\DeclareRobustCommand\widecheck[1]{{\mathpalette\@widecheck{#1}}}
\def\@widecheck#1#2{%
    \setbox\z@\hbox{\m@th$#1#2$}%
    \setbox\tw@\hbox{\m@th$#1%
       \widehat{%
          \vrule\@width\z@\@height\ht\z@
          \vrule\@height\z@\@width\wd\z@}$}%
    \dp\tw@-\ht\z@
    \@tempdima\ht\z@ \advance\@tempdima2\ht\tw@ \divide\@tempdima\thr@@
    \setbox\tw@\hbox{%
       \raise\@tempdima\hbox{\scalebox{1}[-1]{\lower\@tempdima\box
\tw@}}}%
    {\ooalign{\box\tw@ \cr \box\z@}}}
\makeatother

\newcommand{\ve}{\varepsilon}

\begin{document}

\title[Class numbers in continued fraction families of real quadratic fields]
{Distribution of class numbers in continued fraction families of real quadratic fields}
\author{Alexander Dahl}
\address{Department of Mathematics and Statistics,
York University,
4700 Keele Street,
Toronto, ON,
M3J1P3
Canada}
\email{aodahl@yorku.ca}

\thanks{V.K. was partly supported by Charles University Mobility Fund and a grant of Deutsche Forschungs\-gemeinschaft.}

\author{V\'\i t\v ezslav Kala}
\address{Mathematisches Institut, Bunsenstr.~3-5, D-37073 G\"ottingen, Germany}
\address{Department of Algebra, Faculty of Mathematics and Physics, Charles University, Sokolov\-sk\' a 83, 18600 Praha~8, Czech Republic}
\email{vita.kala@gmail.com}


\date{\today}


\begin{abstract}
We construct a random model to study the distribution of class numbers in special families of real quadratic fields $\mathbb Q(\sqrt d)$ arising from continued fractions. 
These families are obtained by considering continued fraction expansions of the form $\sqrt {D(n)}=[f(n), \overline{u_1, u_2, \dots, u_{s-1}, 2f(n)}]$ with fixed coefficients $u_1, \dots, u_{s-1}$ and generalize well-known families such as Chowla's $4n^2+1$, for which analogous results were recently proved by Dahl and Lamzouri \cite{DL}.
\end{abstract}

\maketitle

\section{Introduction}

The class number of an algebraic number field measures to what degree unique factorization fails. The study of this invariant for quadratic number fields $\mathbb{Q}(\sqrt{d})$ where $d$ is a fundamental discriminant spans centuries, starting most seriously with the work of Gauss. A question of great interest is, what is the distribution of class numbers among quadratic fields?

For a discriminant $d$, define $h(d)$ to be the class number. For imaginary quadratic fields, Gauss conjectured that $h(d)$ increases without bound with $-d$, which was later proven by Heilbronn. Gauss also posed his famous class number problem: given a class number $h$, for how many discriminants $d$ do we have $h(d)=h$? This was settled for $h=1$ by Heegner, Baker, and Stark, and all imaginary quadratic fields have been catalogued for $h\leq 100$ due to the work of Watkins \cite{Wa}.

The situation in the case of positive discriminants is more difficult, however. Gauss conjectured that there are an infinite number of real quadratic fields with class number $1$, but this is still not known. The difficulty is reflected by Dirichlet's class number formula in this case, which is
\begin{equation}
\label{eq:dirichlet-class-number-formula}
	h(d)=\frac{L(1,\chi_d) \sqrt{d}}{\log \varepsilon_d},
\end{equation}
where $\varepsilon_d$ is the fundamental unit, which generates the (infinite) group of units. We have $\varepsilon_d = (a+b\sqrt{d})/2$, where $a$ and $b$ are the smallest positive integer solutions to the Pell equations $a^2-b^2 d=\pm 4$. The issue here is that, although the size of $L(1,\chi_d)$ is relatively stable, $ \varepsilon_d $ fluctuates considerably throughout a very large range. One strategy for dealing with this problem is to restrict study to specific families of discriminants where the fundamental units are controlled.

Two such examples are Yokoi's and Chowla's families of real quadratic fields $\mathbb Q(\sqrt{m^2+4})$ and $\mathbb Q(\sqrt{4m^2+1})$. In these families, the fundamental unit $\varepsilon_d$ is exactly given and is as small as possible, namely of order $\sqrt{d}$, producing large class numbers. 
The distribution of class numbers in Chowla's family was recently studied in \cite{DL}, and we direct the reader there for more details and a discussion on its relation to the upper bound for the size of the class number. 
The distribution of class numbers follows from a study of the distribution of $L(1,\chi_d)$ over $d$ in the family using a random model $L(1,\mathbb{X})$, where $\mathbb{X}$ is a random variable which models the behaviour of $\chi_d$.

We show that these families belong to a larger class of families whose fundamental units are likewise as small as possible, and therefore for which the same distribution results in \cite{DL} hold. The idea is that, if for example a discriminant $d$ is 1 modulo 4, then we have a repeated continued fraction expansion of the form
\[
	\omega_d := \frac{1+\sqrt{d}}{2} = [u_0, \overline{u_1, u_2, \dots, u_{s-1}, 2u_0-1}],
\]
where the list of coefficients $u_1, \dots, u_{s-1}$ is symmetric. We now have $ \varepsilon_d = \alpha_{s-1}$, where $\alpha_n$ is a sequence corresponding to the convergents of this continued fraction (see \S \ref{sec:continued fractions}), and by Lemma \ref{lem:estimate epsilon}, we get a good estimate for $\varepsilon_d$.

Yokoi's family, which is the set of squarefree discriminants of the form $m^2+4$, fits into this framework in the following way: For $d=m^2+4$, we must have $d$ odd, and if we hence recast the family as discriminants of the form $d=(2n+1)^2+4$, then we see that
\[
	\omega_{d(n)}=[n+1;\overline{2n+1}].
\]
The idea is to generalize this by considering discriminants $d$ with
\[
	\omega_{d} = [f(n), \overline{u_1, u_2, \dots, u_{s-1}, 2f(n)-1}]
\]
for some linear polynomial $f(n)$. It turns out that in this case, solutions for $d$ are given by values of a quadratic polynomial $D(n)$ and yield a ``continued fraction family" of real quadratic fields, in which we can study the distribution of class numbers.

\

Such families are a well-studied subject. The conditions on the coefficients $u_1, \dots, u_{s-1}$ under which there are infinitely many squarefree $d$ were obtained by Friesen \cite{Fr} and Halter-Koch \cite{HK} (and others, see the beginning of \S \ref{sec:continued fractions} for more information).

As we already mentioned, in these families we have that $\varepsilon_{D(n)}\asymp \sqrt{D(n)}$. 
This was already used in a number of previous works, for example Kawamoto and Tomita \cite{KT} show that if the class number of $\mathbb Q(\sqrt {D(n)})$ is one, then $n$ has to be the smallest possible, up to at most 52 exceptions. The class number one problem in special families of real quadratic fields was
considered and solved in various cases in the works of Louboutin \cite{Lo}, Bir\' o, Granville, and Lapkova \cite{Bi}, \cite{BG}, \cite{BL}, \cite{La}
(also see Mollin \cite{Mo} for a survey of earlier results). Such families were also recently considered by Blomer and Kala in the context of universal quadratic forms (and indecomposable integers) over real quadratic number fields \cite{BK}, \cite{Ka}, \cite{Ka2}.

Conversely, in Proposition \ref{prop:epsilon bound implies contfrac} we show that for a polynomial $D(n)$, the assumption that the fundamental unit is bounded by $\sqrt{D(n)}$ essentially implies that $D(n)$ comes from a continued fraction. Related questions focusing on the length of the period of the continued fraction expansion of $\sqrt{D(n)}$ for a polynomial $D(n)$ have been studied by a number of people, from Schinzel \cite{S1}, \cite{S2} to van der Poorten and Williams \cite{vPW}.

There are also constructions of more general families of continued fractions and estimates of the growth of the fundamental unit in them -- let us just mention McLaughlin \cite{McL} and refer the reader there for further references.

\

As we already indicated, we extend the results from \cite{DL} to families arising from continued fractions with (symmetric) constant coefficients $u_1, \dots, u_{s-1}$, which characterize the discriminants by a corresponding quadratic polynomial $D(n)$.
It turns out that the construction for the random model for $L(1,\mathbb{X})$ can be generalized to a random model $L(1,\mathbb{X}_D)$ (cf. \S \ref{sec:random-model}), because its defining feature is the behaviour of the Jacobsthal sum
\[
	\sum_{n=1}^{p} \left( \frac{D(n)}{p} \right) = -1.
\]

We consider both the cases of continued fraction for $\sqrt{d}$ and $(1+\sqrt d)/2$ (i.e., $d\equiv 2, 3\pmod 4$ and $d\equiv 1\pmod 4$, respectively), and obtain the following main results, which  resemble those in \cite{DL}.

In Definition \ref{def:family}, we define a polynomial of {\it continued fraction discriminant type} to be one that arises from a continued fraction expansion.
For such a polynomial, let us now denote the family of continued fraction discriminants arising from $D$ by
\begin{equation}
\label{eq:D_D def}
	\mathcal{D}_D=\lbrace D(n) \mid n\in \mathbb{N}, D(n) \text{ squarefree}\rbrace,
\end{equation}
and define
\[
	\mathcal{D}_D(x)=\lbrace d \in \mathcal{D}_D, d \leq x \rbrace.
\]

Of interest is the maximum and minimum values that $h(d)$ can take in terms of $d$. Using bounds for $L(1,\chi_d)$ obtained by Littlewood \cite{Li} on GRH and the fact that for a positive discriminant $d$ we have $\varepsilon_d \geq \sqrt{d}/2$, we obtain the bounds
\begin{equation}
\label{eq:h-bounds}
	(e^{-\gamma}\zeta(2)+o(1))\frac{\sqrt{d}}{\log d \log \log d}
	\leq
	h(d)
	\leq
	(4e^\gamma+o(1))\frac{\sqrt{d}}{\log d}\log \log d.
\end{equation}
However, the conjectured bounds
\begin{equation}
\label{eq:L-bounds-conj}
	(e^{-\gamma} \zeta(2)+o(1))/\log \log \vert d \vert
	\leq
	L(1,\chi_d)
	\leq
	(e^\gamma+o(1)) \log \log \vert d \vert
\end{equation}
(cf. \cite{GS}) suggest that in fact (\ref{eq:h-bounds}) holds with a lower bound twice as large and an upper bound half as large. It was shown in \cite{Lam1} using Chowla's family that there are at least $x^{1/2-1/\log \log x}$ real quadratic fields with discriminant $d\leq x$ such that the upper bound in (\ref{eq:L-bounds-conj}) is exceeded, and it was also shown that there are no more than $x^{1/2+o(1)}$ such $d$. Further supporting the conjectured bounds is the following theorem, generalized from the case of Chowla's family in \cite{DL}, which states that the tail of large and small values of $h(d)$ over $d\in \mathcal{D}_D$ is double exponentially decreasing.

\begin{theorem}
\label{thm:MainResult}
Let $D$ be a polynomial of continued fraction discriminant type. Let $x$ be large, and $1\leq \tau\leq \log_2 x-3\log_3x$. The number of discriminants $d\in \mathcal{D}_D(x)$ such that 
$$ h(d)\geq 2e^{\gamma}\frac{\sqrt{d}}{\log d} \cdot \tau,$$
equals 
\begin{equation}\label{eq:TailDistributionClass} |\D| \cdot  \exp\left(-\frac{e^{\tau-C_0}}{\tau}\left(1+ O\left(\frac{1}{\tau}\right)\right)\right),
\end{equation}
where 
\begin{equation}\label{eq:SpecialConstant}
C_0:= \int_0^1\frac{\tanh(t)}{t}dt + \int_1^{\infty}\frac{\tanh(t)-1}{t}dt=0.8187\cdots.
\end{equation}
 Moreover, the same estimate holds for the number of discriminants $d\in \D$ such that 
$$ h(d)\leq 2e^{-\gamma}\zeta(2)\frac{\sqrt{d}}{\log d}\cdot \frac{1}{\tau},$$
in the same range of $\tau$. 
\end{theorem}

Let $D$ be a polynomial of continued fraction discriminant type. For $\tau>0$, define 
$$ \Phi_{\X}(\tau):= \mathbb{P}\big(L(1,\X)>e^{\gamma}\tau\big) \text{ and } \Psi_{\X}(\tau):= \mathbb{P}\left(L(1,\X)<\frac{\zeta(2)}{e^{\gamma}\tau}\right).
$$
The next theorem states that the distribution of $L(1,\mathbb{X}_D)$ closely approximates that of $L(1,\chi_d)$ for $d\in\mathcal{D}_D$.

\begin{theorem}
\label{thm:Distribution}
Let $D$ be a polynomial of continued fraction discriminant type, and let $x$ be large. Uniformly in the range $1\leq \tau\leq \log_2 x-2\log_3x-\log_4 x$, we have 
$$\frac{1}{|\D|}\big|\{d\in \D: L(1,\chi_d) >e^{\gamma}\tau \}\big|= \Phi_{\X}(\tau)\left(1+O\left(\frac{e^{\tau}(\log_2 x)^2\log _3 x}{\log x}\right)\right),$$
and 
$$ \frac{1}{|\D|}\left|\left\{d\in \D: L(1,\chi_d) <\frac{\zeta(2)}{e^{\gamma}\tau} \right\}\right|= \Psi_{\X}(\tau)\left(1+O\left(\frac{e^{\tau}(\log_2 x)^2\log _3 x}{\log x}\right)\right).$$
\end{theorem} 
Here and throughout the paper, we denote by $\log_k x$ the $k$-fold iterated logarithm; i.e., $\log_2 = \log \log $, etc. 

\

Finally, in order to deduce Theorem \ref{thm:MainResult}, we need to determine what the behaviour of the distribution of $L(1,\mathbb{X}_D)$ is. This is the content of the following theorem.
\begin{theorem}
\label{thm:ExponentialDecay}
Let $D$ be a polynomial of continued fraction discriminant type. For large $\tau$ we have  
\begin{equation}\label{eq:AsympLargeDeviations}
\Phi_{\X}(\tau)=\exp\left(-\frac{e^{\tau-C_0}}{\tau}\left(1+ O\left(\frac{1}{\tau}\right)\right)\right),
\end{equation}
where $C_0$ is defined in \eqref{eq:SpecialConstant}. 
The same estimate also holds for $\Psi_{\X}(\tau)$. Moreover, if $0\leq \lambda\leq e^{-\tau}$, then we have 
\begin{equation}\label{eq:Perturbations}
\Phi_{\X}\left(e^{-\lambda} \tau\right)=\Phi_{\X}(\tau) \big(1+O\left(\lambda e^{\tau}\right)\big), \textup{ and } \Psi_{\X}\left(e^{-\lambda} \tau\right)=\Psi_{\X}(\tau) \big(1+O\left(\lambda e^{\tau}\right)\big).
\end{equation}

\end{theorem}

As in \cite{DL}, we use these distribution results to examine the number of discriminants in the family $\mathcal{D}_D$ with class number $h$, denoted by $\mathcal{F}_D(h)$. The difference is in  the constant which now depends on the choice of polynomial $D$.

\begin{theorem}
\label{thm:F-average}
Let $D$ be a polynomial of continued fraction discriminant type. As $H \rightarrow \infty$, we have
\[
	\sum_{h\leq H} \F = C_2 H \log H 
	+ O\left(H(\log_2 H)^2\log_3 H\right),
\]
where
\[
	C_2=\frac{1}{\sqrt{a}} \left( 1-\frac{c(4)}{4}\right) \prod_{p>2} \left( 1-\frac{c(p)}{p^2} \right) \ex(L(1,\X)^{-1}),
\]
with
\[
	c(n) :=\#\lbrace k \mod n \mid D(k) \equiv 0 \pmod n.
\]
\end{theorem}
We direct the reader to \S \ref{sec:random-model} for more details on $c(n)$, particularly Lemma \ref{lem:c(p)-identity}. To prove this theorem, we shall need the following result on complex moments of $L(1,\chi_d)$ for $d$ in a family.

\begin{theorem}
\label{thm:ComplexMoments}
Let $D$ be a polynomial of continued fraction discriminant type, and let $x$ be large. There exists a positive constant $B$ such that uniformly for all complex numbers $z$ with $|z|\leq B\log x/(\log_2x\log_3 x)$ we have
$$\frac{1}{|\D|}\sumst_{d\in \D} L(1,\chi_d)^z= \ex\left(L(1,\X)^z\right)+
O\left(\exp\left(-\frac{\log x}{20\log_2x}\right)\right),
$$
where $\sumst$ indicates that the sum is over non-exceptional discriminants $d$.
\end{theorem}

The definition of an exceptional discriminant is given in (3.2) of \cite{DL}, and we also refer the reader to Remark 1.5 there for details.

\section{Discriminant families from continued fractions}
\label{sec:continued fractions}

It is a classical fact that the continued fraction expansion of $\sqrt d$ (for squarefree positive integer $d$) is of the form
$\sqrt{d}=[u_0, \overline{u_1,\dots,u_{s-1},u_s=2u_0}]$, where the sequence $u_1,\dots,u_{s-1}$ is symmetric. One can conversely ask for which 
symmetric sequences $u_1,\dots,u_{s-1}$ there are infinitely many suitable values of $u_0$ and $d$. The answer was probably already known to Euler in 1765
and is formulated for example in \cite[\S 26]{Pe}. This was later extended by Friesen \cite{Fr} and Halter-Koch \cite{HK} to include the condition on the 
squarefreeness of $d$ and also the case of continued fractions for $(1+\sqrt d)/2$ (for this case, see also \cite[\S 30]{Pe}).

The ``continued fraction families" that we are considering arise from these considerations. However, let us first recall some basic facts about continued fraction convergents (see, e.g., \cite{Pe}):

\

Let $d$ be a squarefree positive integer, $K=\mathbb Q(\sqrt d)$ and define
$$\omega_d=\begin{cases}
\sqrt d &\mbox{if } d\equiv 2,3\\
\frac{1+\sqrt d}2&\mbox{if } d\equiv 1
\end{cases}\pmod 4.$$
Then
$$-\omega_d\pr=\begin{cases}
\sqrt d &\mbox{if } d\equiv 2,3\\
\frac{-1+\sqrt d}2&\mbox{if } d\equiv 1
\end{cases}\pmod 4,$$
and $\mathcal O_K=\mathbb Z[\omega_d]$.

Let $\omega_D=[u_0, \overline{u_1, u_2, \dots, u_{s-1}, u_s}]$ be the periodic continued fraction expression for $\omega_d$. Note that $u_s=2u_0$ when $d\equiv 2,3\pmod 4$ and $u_s=2u_0-1$ when $d\equiv 1\pmod 4$. Also the sequence $(u_1, u_2, \dots, u_{s-1})$ is symmetric.

Let $\frac {p_i}{q_i}:=[u_0, \dots, u_i]$ be the $i$th convergent to $-\omega_d\pr$, 
and $\alpha_i=p_i-q_i\omega_d\pr$ the corresponding element of $\mathcal O_K$. Then
$p_{i+1}=u_{i+1}p_i+p_{i-1}$, $q_{i+1}=u_{i+1}q_i+q_{i-1}$ and 
$\alpha_{i+1}=u_{i+1}\alpha_{i}+\alpha_{i-1}$
(with initial conditions $p_{-1}:=1$, $p_0=k$, $q_{-1}:=0$, $q_0=1$). Note that $\alpha_{-1}=1$. 

Finally, we have $\ve=\alpha_{s-1}$, where $\ve>1$ is the fundamental unit of $K$. 

\

The following proposition establishes the discriminant family parametrization arising from a continued fraction expansion. We mostly follow Halter-Koch \cite{HK} in the statement.

\begin{prop}[\cite{Fr}, \cite{HK}]
\label{prop:continued fraction discriminants}
Let $u_1,\dots,u_{s-1}$ be a symmetric sequence of positive integers (with $s\geq 1$) and define a sequence $q_i$ for $-1\leq i\leq s$ by the recurrence $q_{i+1}=u_{i+1}q_i+q_{i-1}$,
$q_{-1}=0$, $q_0=1$. Then:

a) The equation
\begin{equation}\label{eq:1}
\sqrt{d}=[k, \overline{u_1,\dots,u_{s-1},2k}],\ \ k=\lfloor \sqrt d\rfloor
\end{equation}
has infinitely many squarefree positive solutions $d\equiv 2,3\pmod 4$ if and only if
$q_{s-1}$ is odd or $q_{s-1}$, $q_{s-2}q_{s-3}$ are both even.

b) The equation
\begin{equation}\label{eq:2}
\frac{1+\sqrt{d}}2=[k, \overline{u_1,\dots,u_{s-1},2k-1}],\ \ k=\left\lfloor \frac{1+\sqrt{d}}2\right\rfloor
\end{equation}
has infinitely many squarefree positive solutions $d\equiv 1\pmod 4$ if and only if
$q_{s-1}$ is odd or $q_{s-1}$, $q_{s-2}q_{s-3}+1$ are both even.

When the corresponding condition is satisfied, all the solutions to \eqref{eq:1} or \eqref{eq:2} are given as
$$d=D(n)=an^2+bn+c,\ \ k(n)=en+f$$
for $n\in \mathbb{N}$, where the integers $a$, $b$, $c$, $e$, and $f$ depend on $u_1,\dots,u_{s-1}$, $a\neq 0$ is a square, $e\neq 0$, and $b^2-4ac=(-1)^s$ or $(-1)^s\cdot 4$ or $(-1)^s\cdot 16$.

Also, if $D(n)$ is squarefree, and $D(n)\equiv 2, 3\pmod 4$ in case a) and $D(n)\equiv 1\pmod 4$ in case b), then 
the fundamental unit of the quadratic field $\mathbb Q(\sqrt {D(n)})$ is of the form
\[
	\varepsilon_{D(n)} = p(n)+q\sqrt{D(n)} \mathit{\ \ or\ \ } p(n)+q\cdot\frac{1+\sqrt{D(n)}}2
\]
where $p$ is a linear polynomial, $q=q_{s-1}$ is a constant, and both $p$ and $q$ depend on $u_1,\dots,u_{s-1}$.
\end{prop}

\begin{proof}
These are essentially \S 2 and \S 5 in \cite{HK}. Note that $e,f,g$ from that paper correspond to our $q_{s-1}, q_{s-2}, q_{s-3}$.
In \cite{HK}, the admissible values of the variable $n$ are implicitly taken as $n>n_0$ for some $n_0$, which we can modify to $n>0$ by the substitution $n\mapsto n-n_0$. 

The only part that does not appear in \cite{HK} is the additional condition that $d\equiv 2, 3\pmod 4$ in part a). This is not hard to verify by a case-by-case analysis of $d\pmod 4$, depending on the parities of $q_{s-1}, q_{s-2}, q_{s-3}$. Note that due to the recurrence, $q_{s-1}, q_{s-2}$ can't both be even, and likewise $q_{s-2}, q_{s-3}$ can't both be even.

The formula for the discriminant $b^2-4ac$ can be directly verified using the explicit values of $a,b,c$ stated in \cite{HK} and the fact that $q_{s-2}^2=q_{s-1}q_{s-3}+(-1)^s$.

The final statement concerning the fundamental unit is classical and for example appears as Theorems 3.18 and 3.35 in \cite{Pe}.
\end{proof}

For our families of real quadratic fields we wish to consider $\mathbb Q(\sqrt {D(n)})$ for $D(n)$ as above. However, we need to control the behaviour of the fundamental unit in such a family, which depends on whether $D(n)\equiv 1$ or $\equiv 2, 3\pmod 4$. 
Unfortunately, it sometimes happens that in the
case a) of Proposition \ref{prop:continued fraction discriminants}, the polynomial $D(n)$ attains also values congruent to $1\pmod 4$: for example when 
$\sqrt {D(n)}=[n, \overline{2n}]$, then $D(n)=n^2+1\equiv 1, 2\pmod 4$.  But when $D(n)\equiv 1\pmod 4$ one should consider the continued fraction expansion of $\frac{1+\sqrt {D(n)}}2$ instead. 
However, the proposition guarantees the existence of at least one $n_0$ such that $D(n_0)\equiv 2, 3\pmod 4$. Then we have $D(4m+n_0)\equiv D(n_0)\pmod 4$ for all $m$, and so the subfamily $D(4m+n_0)$ attains only the desired values mod 4. 

Somewhat surprisingly, this problem does not occur for families that arise in case b): all the values of $D(n)$ then are $\equiv 0, 1\pmod 4$, as can be easily seen by the explicit formula in \cite{HK}. We can allow values that are congruent to $0\pmod 4$ for now, as they will be sieved out by the squarefree condition in (\ref{eq:D_D def}).

This discussion motivates the following definition:

\begin{definition}\label{def:family}
Let $D(n)$ be a polynomial arising from a continued fraction equation as in Proposition  \ref{prop:continued fraction discriminants} and let $N(n)=un+v$ with $u=2^i$ for $i\geq 0$ and $-u<v\leq 0$.
If $D(n)$ comes from the case a) of Proposition \ref{prop:continued fraction discriminants}, also assume that we have $D(N(n))\equiv 0, 2, 3\pmod 4$ for all $n$. Then we say that $D(N(n))$ is a \textit{polynomial of continued fraction discriminant type}.
\end{definition}

The leading coefficient of every polynomial of continued fraction discriminant type is a square and its discriminant is $\pm$ a power of two, which we will need later. Also, as we discussed before the definition, for every $D(n)$ from Proposition \ref{prop:continued fraction discriminants}, there is some $N(n)$ such that $D(N(n))$ is of continued fraction discriminant type (and it suffices to take $u=4$ in case a) and $u=1$ in case b)).

\

The formulas for $D(n)$ and $\varepsilon_{D(n)}$ in Proposition \ref{prop:continued fraction discriminants} immediately imply the fact that $\varepsilon_{D(n)}\asymp\sqrt{D(n)}$ (with the implied constant depending on the polynomial $D(n)$). However, we can also prove more precise information on the size of the fundamental unit
$\ve=\alpha_{s-1}$; part b) is useful to have a non-trivial lower bound when many of the coefficients $u_i$ are equal to 1.

\begin{lemma}\label{lem:estimate epsilon}
a) $\prod_{i=0}^j u_i<\alpha_j<\prod_{i=0}^j (u_i+1)$

b) $u_0u_{j}^e\prod_{i=1}^{\lfloor j/2\rfloor} (u_{2i}u_{2i-1}+1)<\alpha_j$, where $e=0$ when $j$ is even and $e=1$ when $j$ is odd.
\end{lemma}

\begin{proof}
Part a) easily follows by induction from the recurrence $\alpha_{i+1}=u_{i+1}\alpha_{i}+\alpha_{i-1}$.

For b), we have 
$$\alpha_{i+1}=u_{i+1}\alpha_i+\alpha_{i-1}=(u_{i+1}u_i+1)\alpha_{i-1}+u_{i+1}\alpha_{i-2}.$$
The estimate then again follows by induction.
\end{proof}

We see that in a continued fraction family, we have a good control on the size of $\ve$ in terms of the variable $n$, and so also in terms of $D(n)$ (note that in Lemma \ref{lem:estimate epsilon}, $u_0=k(n)$ is a linear polynomial in $n$). Let us now consider the converse problem: 
Does every polynomial family of real quadratic fields with $\varepsilon_{D(n)}\ll {D(n)}^{1/2}$ arise from a continued fraction as in Proposition \ref{prop:continued fraction discriminants}? As we will soon see, the answer is (essentially) yes.

Assume that $D(X)\in\mathbb Z[X]$ is a polynomial and 
let $\ve(x)$ be the fundamental unit of $\mathbb Q(\sqrt {D(x)})$ for $x\in\mathbb N$ with $D(x)>0$.

\begin{prop}\label{prop:epsilon bound implies contfrac}
Assume that for a positive density of $x\in\mathbb N$ we have that $D(x)$ is squarefree and 
$\ve(x)\ll D(x)^{1/2}$. Then there is $k(X)\in\mathbb Q[X]$ and $u_1, \dots, u_{s-1}\in\mathbb Z$ such that 
$$\sqrt {D(x)}=[k(x), \overline{u_1, u_2, \dots, u_{s-1}, 2k(x)}]\mathit{\ or\ } \frac{1+\sqrt {D(x)}}2=[k(x), \overline{u_1, u_2, \dots, u_{s-1}, 2k(x)-1}]$$
for each $x\in\mathbb N$ with $D(x),k(x)\in\mathbb N$.
\end{prop}

\begin{proof}
Let $S_0$ be the set of $x\in\mathbb N$ that satisfy the conditions in the statement.
For $x\in S_0$, each value $D(x)$ is congruent to $1, 2$, or $3$ mod $4$ (because $D(x)$ is squarefree), and so $S_0\cap \{D(x)\equiv 1\pmod 4\}$ or  $S_0\cap \{D(x)\equiv 2,3\pmod 4\}$ has positive density. For the rest of the proof assume that the second case holds (the argument in the first case is almost identical).

Let $k(x):=\lfloor \sqrt {D(x)}\rfloor$ and let $C>0$ be such that $\ve(x)<C\cdot k(x)$ for all $x\in S_1\subset S_0\cap \{D(x)\equiv 2,3\pmod 4\}\subset \mathbb N$, where $S_1$ has positive density.

Let $\sqrt {D(x)}=[k(x), \overline{u_1(x), u_2(x), \dots, u_{s(x)-1}(x), 2k(x)}].$

By Lemma \ref{lem:estimate epsilon}b), we have 
$$Ck(x)\geq\ve(x)>k(x)\prod_{i=1}^{(s(x)-1)/2} (u_{2i}(x)u_{2i-1}(x)+1)>k(x)\prod_{i=1}^{(s(x)-1)/2} 2=k(x)\cdot 2^{(s(x)-1)/2}.$$
Hence $s(x)<c_1$ for some constant $c_1$.

By Lemma \ref{lem:estimate epsilon}a) we similarly see that $u_i(x)<c_2$ for all $i$ and some constant $c_2$.
Hence there are only finitely many possibilities for the tuple $(s(x), u_1(x), u_2(x), \dots, u_{s(x)-1}(x))$, and so one of them must occur for all values of $x\in S\subset S_1$, where $S$ has positive density. 

Thus $\sqrt {D(x)}=[k(x), \overline{u_1, u_2, \dots, u_{s-1}, 2k(x)}]$ for $x\in S$.
Solving this equation as in Proposition \ref{prop:continued fraction discriminants} gives $D(x)=at(x)^2+bt(x)+c$ and $k(x)=et(x)+f$ for some $a, b, c,e,f\in\mathbb Z$. 

We have $4aD(x)=(2at(x)+b)^2+(4ac-b^2)$, which implies that $f(x)=y^2$ has an integral solution for all $x\in S$, where $f(x)=4aD(x)-(4ac-b^2)$ is a polynomial of degree $\deg D$. 

Hence by Lemma \ref{lem:siegel} below, $y(x)=2at(x)+b=\pm g(x)$ for a polynomial $g(X)\in\mathbb Q[X]$. Restricting to a positive density subset of $S$, we get rid of the $\pm$ sign to conclude that $t(x)$ is a polynomial, and hence also $k(X)\in\mathbb Q[X]$.
But then $\sqrt {D(X)}=[k(X), \overline{u_1, u_2, \dots, u_{s-1}, 2k(X)}]$ holds as an identity of polynomials, finishing the proof.
\end{proof}

It remains to show the following lemma:

\begin{lemma}\label{lem:siegel}
Let $f(X)\in\mathbb Q[X]$ be such that the equation $y^2=f(x)$ has a solution $y=y(x)\in\mathbb Z$ for all $x\in T\subset\mathbb Z$, where $T$ has positive density. Then $y(x)=\pm g(x)$ for all $x\in T$, where $g(X)\in\mathbb Q[X]$ and $g(x)^2=f(x)$.
\end{lemma}

\begin{proof}
This is essentially Siegel's theorem on finiteness of integral points on varieties of positive genus.
Let us give more details by following the proof of Theorem 4.3 in Silverman \cite{Si}. As in the proof, let $K$ be a number field with ring of integers $R$ and $S$ a finite set of primes of $R$ such that $f$ splits over $K$, i.e., $f(x)=a(x-\alpha_1)^{n_1}\cdots(x-\alpha_d)^{n_d}$ (we are not assuming that $f$ is separable) and conditions (i) -- (iii) from the proof are satisfied. Moreover, let $m_i:=\lfloor \frac{n_i}2\rfloor$ and $e_i:=n_i-2m_i$. Assume that $e_1=\dots=e_k=1$ and $e_{k+1}=\dots=e_n=0$.

Let $x, y\in\mathbb Z$ be such that $y^2=f(x)$. Then $y$ is divisible by $(x-\alpha_1)^{m_1}\cdots(x-\alpha_d)^{m_d}$, and so we can write 
$y=(x-\alpha_1)^{m_1}\cdots(x-\alpha_d)^{m_d}z$ for some $z\in R_S$. Then we have $z^2=a(x-\alpha_1)\cdots(x-\alpha_k)$. 

If $k=0$, then $z=\pm\sqrt a$ and $y$ is given by a polynomial in $x$. Due to the positive density assumption, this polynomial must have rational coefficients.

If $k=1$, the equation $z^2=a(x-\alpha_1)$ can't have a solution in $R_S$ for a positive density of $x\in\mathbb Z$, a contradiction.
Likewise if $k=2$, in which case the equation $z^2=a(x-\alpha_1)(x-\alpha_2)$ is essentially a Pell equation.

Finally,
if $k\geq 3$, then $z^2=a(x-\alpha_1)\cdots(x-\alpha_k)$ has only finitely many solutions by Theorem 4.3 in \cite{Si}.
\end{proof}

\section{Random model}
\label{sec:random-model}

Suppose that $D(n)=an^2+bn+c$ is of continued fraction discriminant type, defined in Definition \ref{def:family}. Recall that in particular $a$ is a square and that its discriminant $\Delta=b^2-4ac$ is $\pm$ a power of two.

Completing the square, we see that
\begin{equation}
\label{eq:D(n)-complete-the-square}
	D(n)=a(n+b/2a)^2-(b^2/4a-c).
\end{equation}

For an integer $n$, define
\[
	c(n)=\#\lbrace k \mod n \mid D(k) \equiv 0 \pmod n \rbrace.
\]

We then have the identity
\begin{lemma}\label{lem:c(p)-identity}
For an odd prime $p$, we have
\[
	c(p)= \begin{cases}
		1+\left( \frac{\Delta}{p} \right), & p \nmid a; \\
		1, & p \mid a \text{ and } p \nmid b; \\
		0, & p \mid (a,b) \text{ and } p \nmid c; \\
		p, & p \mid (a, b, c).
	\end{cases}
\]
\end{lemma}

\begin{proof}
For prime $p\nmid a$, we have
\[
	D(n)\equiv 0 \pmod p
\]
if and only if
\[
	(n+b/2a)^2 \equiv (b^2/4a^2-c/a) \pmod p,
\]
so the number of solutions is
\[
	1+\left( \frac{b^2/4a^2-c/a}{p} \right)
	=1+\left( \frac{\Delta/(2a)^2}{p} \right)
	= c(p).
\]
The cases where $p\mid a$ follow trivially.
\end{proof}
We note that for odd prime $p$, except for the case where $p\mid (a,b,c)$, we have $c(p^2)=c(p)$.

We have the known Jacobsthal identity
\begin{equation}\label{jacobsthal}
	\sum_{n=0}^{p-1} \left( \frac{n^2+b}{p} \right)
	=
	\begin{cases}
		-1, & p\nmid b; \\
		p-1, & p \mid b.
	\end{cases}
\end{equation}
For an odd prime $p$, we define
\[
	J(p)=\sum_{n=1}^{p} \left( \frac{D(n)}{p} \right),
\]
and we extend this definition multiplicatively to any odd squarefree integer, and hence set $J(1)=1$.

We have the identity

\begin{lemma}
For an odd prime $p$ we have
\[
	J(p)
	= \begin{cases}
		-1, & p\nmid a \text{ and } p\nmid \Delta; \\
		(p-1), & p\nmid a \text{ and } p\mid \Delta; \\
		\left( \frac{c}{p} \right) (p-1), & p\mid (a,b) \text{ and } p\nmid c; \\
		0, & p\mid a \text{ and } p\nmid b, \text{ or } p \mid (a,b,c).
	\end{cases}
\]
\end{lemma}

\begin{proof}
First, we assume that $p\mid a$. Using (\ref{eq:D(n)-complete-the-square}), we have
\begin{align*}
	J(p)=\sum_{n=0}^{p-1} \left( \frac{D(n)}{p} \right)
	= & \left( \frac{a}{p} \right) \sum_{n=0}^{p-1} \left( \frac{(n+b/2a)^2-(b^2/4a^2-c/a)}{p} \right) \\
	= & \left( \frac{a/4}{p} \right) \sum_{n=0}^{p-1} \left( \frac{(2an)^2-\Delta}{p} \right),
\end{align*}
and the result holds by the Jacobsthal identity (\ref{jacobsthal}) above.

The cases where $p\mid a$ follow trivially.
\end{proof}

Since $\Delta$ is plus or minus a power of 2, we have the following corollary.
\begin{cor}
\label{cor:character sum}
For an odd prime $p$ we have
\[
	J(p) = \begin{cases}
		-1, & p\nmid a; \\
		\left( \frac{c}{p} \right)(p-1), & p \mid (a,b), p\nmid c; \\
		0, & p \mid a \text{ and } p\nmid b, \text{ or } p\mid (a,b,c).
	\end{cases}
\]
\end{cor}

We are now in a position to define our random model. Let $\{\mathbb{X}_D(p)\}_p$ be a sequence of independent random variables indexed by the odd primes, and taking the value $1$ with probability $\alpha_p$, $-1$ with probability $\beta_p$, and $0$ with probability $\gamma_p$.

Let $p$ be an odd prime. If $D(n)$ is squarefree then $D(n)$ lies in one of $p^2-c(p^2)$ residue classes modulo $p^2$, since $p^2\nmid D(n)$.  Among these, $\chi_d(p)=0$ for exactly $pc(p)-c(p^2)$ of them, so we define
\[
	\gamma_p=\frac{p c(p)-c(p)}{p^2-c(p)}
	=1-\left( 1-\frac{c(p)}{p} \right) \left( 1-\frac{c(p)}{p^2} \right)^{-1}.
\]
As for $\alpha_p$ and $\beta_p$, they are determined by the equations
\[
	\alpha_p+\beta_p=1-\gamma_p
\]
and
\[
	\alpha_p-\beta_p = \frac{J(p)}{p} \left( 1- \frac{c(p^2)}{p^2} \right)^{-1}.
\]
Now for each $\ell\geq 1$ let $\X(2^\ell)$ be a random variable that is $1$ with probability $\alpha_{2^\ell}$, $-1$ with probability $\beta_{2^\ell}$, and 0 with probability $\gamma_{2^\ell}$. We define
\[
	\gamma_{2^\ell}=\begin{cases}
		c(2)/2, & \ell = 1; \\
		0, & \ell>1.
	\end{cases}
\]
We also define
\[
	\alpha_{2^\ell}+\beta_{2^\ell}=1-\gamma_{2^\ell}
\]
and
\[
	\alpha_{2^\ell} - \beta_{2^\ell} = K(\ell) \left( 1-\frac{c(4)}{4} \right)^{-1},
\]
where
\[
	K(\ell)= \frac{1}{2^\ell} \sum_{n=1}^{2^\ell} \left( \frac{D(n)}{2} \right)^\ell.
\]
For $\ell\geq 1$, we have the identity
\[
	K(\ell)=\begin{cases}
		1-c(2)/2, & \ell \text{ even}; \\
		\frac{1}{\min(2^\ell,8)}\sum_{n=1}^{\min(2^\ell,8)} \left( D(n)/2 \right), & \ell \text{ odd},
	\end{cases}
\]
where the summand in the second case is the Legendre symbol.
Hence, we see that $\X(2^{2\ell+2})=\X(4)$ and $\X(2^{2\ell+3})=\X(8)$ for all $\ell\geq 0$.

Finally, we define $\X(1)=1$, and for any positive integer $m=2^\ell p_1^{a_1}\cdots p_r^{a_r}$ with the $p_i$'s odd, we define the random variable $\X(m)=\X(2^\ell) \X(p_1)^{a_1} \cdots \X(p_r)^{a_r}$.

\section{Modelling $\chi_d(m)$ over $d$ in a family by the random model}

We relate the average value of $\chi_d(m)$ for $d\in \mathcal{D}_D(x)$ to $\ex (\mathbb{X}_D(m))$ via the following proposition.
\begin{prop}
\label{prop:average of character-ex}
Let $D(n)=an^2+bn+c$ be of continued fraction discriminant type. We have
\[
	\vert \mathcal{D}_D(x) \vert = y \left( 1-\frac{c(4)}{4} \right) \prod_{p>2} \left( 1-\frac{c(p)}{p^2} \right) + O(x^{1/3} \log x),
\]
where
\[
	y=\frac{\sqrt{x+\Delta/4a}}{\sqrt{a}}-\frac{b}{2a},
\]
and for a positive integer $m$, we have
\[
	\frac{1}{\vert \mathcal{D}_D(x) \vert} \sum_{d\in \mathcal{D}_D(x)} \chi_d(m) = \ex(\X(m))+O(m^{2/3} x^{-1/6} \log x).
\]
\end{prop}

We shall need two lemmas to prove this result. Analogous to Lemma 2.2 of \cite{DL}, we have

\begin{lemma}
\label{lem:Ex as product}
Let $D$ be a polynomial of continued fraction discriminant type. Let $m=2^\ell p_1^{a_1}\cdots p_k^{a_k}$ be the prime factorization of $m$, and let $m_0$ be the squarefree part of $m/2^\ell$. Then we have
$$
\ex(\X(m))=
	\frac{K(\ell)J(m_0)}{m_0}
	\prod_{\substack{1\leq j\leq k\\2\mid a_j}}
	\left(1-\frac{c(p_j)}{p_j}\right)\
	\left( 1-\frac{c(4)}{4} \right)^{-\sigma(\ell)}
	\prod_{j=1}^k\left(1-\frac{c(p_j)}{p_j^2}\right)^{-1},
$$
where
\[
	\sigma(\ell)=\begin{cases}
		1, & \ell >0, \\
		0, & \ell = 0.
	\end{cases}
\]
\end{lemma}
\begin{proof}
Using the independence of the $\X(p)$'s we obtain
\begin{equation}\label{Multiplicative}
\ex(\X(m))= \ex(\X(2^\ell)) \prod_{j=1}^k\ex\left(\X(p_j)^{a_j}\right).
\end{equation}
First, if $a_j$ is even then
$$ \ex\left(\X(p_j)^{a_j}\right)= \alpha_{p_j}+\beta_{p_j}=1-\gamma_{p_j}
=
\left(1-\frac{c(p_j)}{p_j}\right)\left(1-\frac{c(p_j)}{p_j^2}\right)^{-1}.
$$
On the other hand, if $a_j$ is odd then
$$\ex\left(\X(p_j)^{a_j}\right)=\alpha_{p_j}-\beta_{p_j}=\frac{J(p_j)}{p_j}\left(1-\frac{c(p_j)}{p_j^2}\right)^{-1}.
$$
Finally, we have
\[
	\ex(\X(2^\ell)) = K(\ell) \left( 1-\frac{c(4)}{4} \right)^{-\sigma(\ell)}.
\]
Inserting these estimates in \eqref{Multiplicative} completes the proof.
\end{proof}

Analogous to Lemma 2.3 of \cite{DL}, we have
\begin{lemma}
\label{lem:charsum decomposition}
	Let $D$ be a polynomial of continued fraction discriminant type. Let $m=2^\ell p_1^{a_1}\cdots p_k^{a_k}$ be the prime factorization of $m$, and let $m_0$ be the squarefree part of $m/2^\ell$. Then we have
\[
	\frac{1}{m}\sum_{n=1}^m \left( \frac{D(n)}{m} \right)
	= \frac{K(\ell)J(m_0)}{m_0}
	\prod_{\substack{1\leq j \leq k \\ 2 \mid a_j}} \left( 1- \frac{c(p_j)}{p_j} \right).
\]
\end{lemma}

\begin{proof}  Observe that the sum $\sum_{n=1}^m\left(D(n)/m\right)$ is a complete character sum, and hence by multiplicativity and the Chinese remainder theorem, we have 
\[
\sum_{n=1}^{m}\left(\frac{D(n)}{m}\right)
=K(\ell)
\prod_{j=1}^{k}\left(\sum_{n_{j}=1}^{p_{j}^{a_j}}\left(\frac{D(n_j)}{p_{j}}\right)^{a_j}\right).
\]
If $a_j$ is even then
\[
\sum_{n_j=1}^{p_{j}^{a_j}}\left(\frac{D(n_j)}{p_j}\right)^{a_j}=p_j^{a_j-1}(p_{j}-c(p_{j}))
=p^{a_j}\left( 1-\frac{c(p_j)}{p_j} \right),
\]
by the definition of $c(p)$.

On the other hand, if $a_j=2b_j+1$ is odd then
\[
\sum_{n_{j}=1}^{p_{j}^{a_{j}}}\left(\frac{D(n_{j})}{p_{j}}\right)^{a_{j}}=\sum_{n_{j}=1}^{p_{j}^{a_{j}-1}}\sum_{d=1}^{p_{j}}\left(\frac{D(n_{j}p_{j}+d)}{p_{j}}\right)^{2b_{j}+1}=p_{j}^{a_{j}-1}J(p_j).
\]
Applying Corollary \ref{cor:character sum} completes the proof.
\end{proof}
 
\begin{proof}[Proof of Proposition \ref{prop:average of character-ex}]
To simplify our notation, we define $S(x)=\sum_{d\in \D} \chi_d(m)$.
Then, using that $\mu^2(n)=\sum_{r^2\mid n} \mu(r)$ we obtain
\begin{align*}
S(x)&=\sum_{n\leq y}\left(\frac{D(n)}{m}\right)\mu^{2}(D(n))
= \sum_{n\leq y}\left(\frac{D(n)}{m}\right)\sum_{r^2\mid D(n)}\mu(r)\\
&= \sum_{\substack{r\leq \sqrt{x}\\
(r, m)=1}}\mu(r)\sum_{\substack{n\leq y\\
r^{2}\mid D(n)}}\left(\frac{D(n)}{m}\right).
\end{align*}
Let $2\leq T\leq y$ be a real parameter to be chosen later. We split the above sum over $r$ into two parts: $r\leq T$ and $T< r\leq \sqrt{x}$. Writing $D(n)=r^2 s$ and noting that this is equivalent to $(2an+b)^2-4asr^2=\Delta$, it follows that the contribution of the second part is 
$$ \ll \sum_{T< r\leq \sqrt{x}} \sum_{\substack{n\leq y\\
r^{2}\mid D(n)}} 1\ll \sum_{s\leq x/T^2}\sum_{\substack{n, r\\
(2an+b)^2-4asr^2=\Delta}} 1.$$
We see that the equation $ (2an+b)^2-4asr^2=\Delta $ is a Pell equation in every case. From the theory of Pell's equation, the number of pairs $(u, v)$ for which $1\leq u\leq U$ and $u^2-sv^2=\pm 1, \pm 4$, is $\ll \log U$ uniformly in $s$. Hence, we deduce that the contribution of the terms $T\leq r\leq \sqrt{x}$ to $S(x)$ is $\ll x(\log x)/T^2.$ Thus, 
\begin{equation}\label{CharSumEsti}
S(x)=\sum_{\substack{r\leq T\\
(r, m)=1
}
}\mu(r)\sum_{\substack{n\leq y\\
r^2\mid D(n)
}
}\left(\frac{D(n)}{m}\right)+O\left(\frac{x\log x}{T^{2}}\right).
\end{equation}

Consider the equation
$
D(n)\equiv0\pmod{r^{2}}. 
$
This congruence has $c(r^2)$ solutions modulo $r^2$. Denote these solutions by $\{a_1,...,a_{c(r^2)}\}$. Then, for any integer $k$ we have
\begin{align*} 
\sum_{\substack{k r^2 m < n\leq (k+1)r^2 m\\
r^{2}\mid D(n)}}\left(\frac{D(n)}{m}\right)
&=\sum_{i=1}^{c(r^2)}\sum_{\substack{k r^2 m < n\leq (k+1)r^2 m\\
n\equiv a_i \text{ mod }r^{2}}}\left(\frac{D(n)}{m}\right)\\
&=\sum_{i=1}^{c(r^2)}\sum_{u=1}^m \left(\frac{D(u)}{m}\right)\sum_{\substack{k r^2 m< n\leq (k+1)r^2 m\\
n\equiv a_i \text{ mod }r^{2}\\ n\equiv u\text{ mod }m}} 1\\
&= c(r^2) \sum_{u=1}^m\left(\frac{D(u)}{m}\right),
\end{align*}
by the Chinese remainder theorem, since $(r, m)=1$.
Therefore, we deduce that
\begin{align*}
\sum_{\substack{n\leq y\\
r^{2}\mid D(n)
}
}\left(\frac{D(n)}{m}\right) 
&= y\frac{c(r^2)}{r^2}\frac{1}{m}\sum_{u=1}^m\left(\frac{D(u)}{m}\right)+ O\big(c(r^2)m\big).\\
\end{align*}
Inserting this estimate in \eqref{CharSumEsti} we get
\[
S(x)=\frac{y}{m}\sum_{u=1}^m\left(\frac{D(u)}{m}\right)
\sum_{\substack{r\leq T\\
(r, m)=1
}
}
\mu(r)
\frac{c(r^2)}{r^2}
+O\left(m\sum_{r\leq T}c(r^2)+\frac{x\log x}{T^2}\right).
\]
Since $c(r^2)\leq 2^{\omega(r)+1}\leq  2d(r)$ (where $d(r)$ is the divisor function), we get $\sum_{r\leq T} c(r^2)\ll T\log T$ and 
\[
\sum_{\substack{r> T\\
(r, m)=1}}\frac{\mu(r)}{r^{2}}c(r^2)\ll \sum_{r>T} \frac{d(r)}{r^2}\ll \frac{\log T}{T},
\]
by using that $\sum_{r\leq t} d(r)\sim t\log t$, together with partial summation. Hence, completing the $r$-sum, the error term is
\[
	O\left(mT\log T+\frac{\sqrt{x}\log T}{T}+\frac{x\log x}{T^2}\right).
\]
Taking $T=(x/m)^{1/3}$, we see that the error term is
\[
	O(m^{2/3} x^{1/3} \log x).
\]
Due to Lemma \ref{lem:charsum decomposition}, the main term is
\begin{equation}
\label{eq:S-main-term}
	\frac{y K(\ell) J(m_0)}{m_0}
	\prod_{\substack{1\leq j \leq k \\ 2 \mid a_j}} \left( 1- \frac{c(p_j)}{p_j} \right)
	\sum_{\substack{r\geq 1 \\ (r,m)=1}}
	\mu(r) \frac{c(r^2)}{r^2}.
\end{equation}
We have the identity
\[
	\sum_{\substack{r=1 \\ (r,m)=1}} \frac{\mu(r)c(r^2)}{r^2}
	= \left( 1-\frac{c(4)}{4} \right)^\sigma \prod_{\substack{p>2 \\ p \nmid m}} \left( 1-\frac{c(p)}{p^2} \right),
\]
where $\sigma=1$ if $2\nmid m$ and vanishes otherwise.
Specializing to $m=1$, we get
\begin{equation}
\label{eq:D-size}
	\vert \mathcal{D}(x) \vert = y \left( 1-\frac{c(4)}{4} \right) \prod_{p>2} \left( 1-\frac{c(p)}{p^2} \right)
	+ O(x^{1/3} \log x).
\end{equation}
For general $m$, using Lemma \ref{lem:Ex as product} on (\ref{eq:S-main-term}), the main term is
\[
	=y \cdot
	\ex(\X(m))
	\left( 1-\frac{c(4)}{4} \right)
	\prod_{p>2} \left( 1- \frac{c(p)}{p^2} \right),
\]
and using (\ref{eq:D-size}) completes the proof. 
\end{proof}

Theorem \ref{thm:ComplexMoments} and its proof matches Theorem 1.4 of \cite{DL}. We outline the basic steps here.

For complex $z$, a primitive Dirichlet character $\chi$, and $\Re(s) > 1$, we have the identity
\[
	L(1,\chi)^z = \sum_{n=1}^\infty \frac{d_z(n)\chi(n)}{n^s},
\]
where $d_z(n)$ is the $z$-fold divisor function. We obtain an approximate functional equation at $s=1$ with a weight function $e^{-n/y}$, where $y$ is roughly the conductor of $\chi$, which holds provided the existence of a certain zero-free rectangle of $L(s,\chi)$ extending to the left of the line $\Re(s) = 1$. This is achieved by noting that the Mellin transform of $e^{-1/y}$ is $\Gamma(-s)$, and moving the contour of the integral
\[
	\frac{1}{2\pi i} \int_{2-i\infty}^{2+i\infty} L(1+s,\chi)^z \Gamma(s) y^s ds = \sum_{n=1}^\infty \frac{d_z(n)\chi(n)}{n^s}
\]
into a zero-free region, obtaining the pole $L(1,\chi)^z$. Hence, the main term contribution in Theorem \ref{thm:ComplexMoments} comes from discriminants $d$ where $L(1,\chi_d)$ has such zero-free regions.

\section{The distribution of $L(1,\chi_d)$ over $d$ in a family: Theorems \ref{thm:MainResult}, \ref{thm:Distribution}, and \ref{thm:ExponentialDecay}}

The rest of the analysis required for the proof of Theorems \ref{thm:Distribution} and \ref{thm:ExponentialDecay} follow that of Theorems 1.2 and 1.3 of \cite{DL}. We refer the reader there for details, and give an outline of the proofs here.

To prove Theorem \ref{thm:ExponentialDecay}, we apply a saddle point analysis, a technique which was used also in \cite{Lam2} to study the distribution of the Euler-Kronecker constant. We sketch an outline of the proof thus: We start with a smooth analogue of Perron's formula to approximate the function $\phi(y)$, defined as the characteristic function on $ y>1$, via a contour integral representing a smooth cut-off function (cf. Lemma 4.7 of \cite{DL}). This was originally developed in \cite{Lam2} (cf. Lemma 5.1 there), and is a slight variation of a formula of A. Granville and K. Soundararajan \cite{GS}. We observe that
$
	\ex[\phi(L(1,\X)(e^\gamma \tau)^{-1})] = \Phi_{\X}(\tau),
$
whence setting $y=L(1,x)(e^\gamma \tau)^{-1}$ and taking the expected value, we now have an approximation for $\Phi_{\X}(\tau)$ in terms of the contour integrals of the form
\[
	\int_{c-i\infty}^{c+i\infty} \ex(L(1,\X)^s)(e^\gamma \tau)^{-s}g(s)ds
	= \int_{c-i\infty}^{c+i\infty} 
	\exp\left( \mathcal{L}(s)-s\gamma-s\log \tau \right)	
	g(s)ds,
\]
where $g(s)$ is a particular weight function roughly of the form $1/\Re(s)-i\Im(s)/\Re(s)^2 + O(t^2/\kappa^3)$, and we define $\mathcal{L}(z)=\log \ex(L(1,\X)^z)$. We now recognize that, for real $s$, the argument in the $\exp$ function in the integral is at a maximum when $\mathcal{L}'(s)=\gamma+\log \tau$. It can be shown that $\mathcal{L}'$ is increasing, and so this equation has a unique solution which we denote by $s=\kappa$. We now move the contour of the integral to $\Re(s)=\kappa$ and express the argument of the $\exp$ function as its second degree Taylor expansion centred at $\kappa$. The term with $(x-\kappa)$ vanishes, and hence, after factoring out constant terms, we see that except for the error term, the integral resembles a Gaussian with standard deviation $(\mathcal{L}''(\kappa))^{-1/2}$ and mean $\kappa$. (In practice, to evaluate these integrals we need to truncate them, and therefore we establish that $\ex(L(1,\X)^{r+it})$ decreases rapidly for $\vert t \vert > \sqrt{r \log r}$ (cf. Lemma 4.6 of \cite{DL})). We hence come to
\[
\Phi_{\X}(\tau)= \frac{\ex\left(L(1,\X)^{\kappa}\right)(e^{\gamma}\tau)^{-\kappa}}{\kappa \sqrt{2\pi \mathcal{L}''(\kappa)}}\left(1+O\left(\sqrt{\frac{\log \kappa}{\kappa}} \right)\right)
\]
(cf. Theorem 4.1 of \cite{DL}), and Theorem \ref{thm:ExponentialDecay} follows after applying estimates for derivatives of $\mathcal{L}$ from Proposition 4.2 of \cite{DL}.

To prove Theorem \ref{thm:Distribution}, we again use the smooth version of Perron's formula in Lemma 4.7 of \cite{DL}: We estimate the distribution function $\Phi_{\X}(\tau)$ by contour integrals of $\ex(L(1,\X)^s$ over $\Re(s)=\kappa$, and we estimate the proportion of $d\leq x$ with $L(1,\chi_d)>e^\gamma \tau $ by contour integrals of $s$-moments of $L(1,\X)$ over the same contour. We then compare the integrals using Theorem \ref{thm:ComplexMoments}.

\begin{proof}[Proof of Theorem \ref{thm:MainResult}]
Let $D(n)$ be a polynomial of continued fraction discriminant type. From Proposition \ref{prop:continued fraction discriminants}, we know that for $d\in \mathcal{D}_D$, we have $\varepsilon_d=P(d)+q\sqrt{d}$ where $P=p \circ D^{-1}$, with $p$ a linear polynomial, and $q$ a constant. We want to estimate the size of
\[
	S(x) := \lbrace d\in\mathcal{D}_D(x) : h(d) \geq 2 e^\gamma \frac{\sqrt{d}}{\log d} \cdot \tau \rbrace
	= \lbrace d\in\mathcal{D}_D(x) : L(1,\chi_d) \geq e^\gamma \tau_0(d) \rbrace
\]
where
\[
	\tau_0(d)= \frac{2 \log \varepsilon_d}{\log d} \cdot \tau
	= \left(1 + \frac{2\log(q+P(d)/\sqrt{d})}{\log d} \right) \cdot \tau.
\]
On the one hand, we see that
\begin{equation}
\label{eq:S(x) upper bound}
	\vert S(x) \vert \leq \vert \lbrace d\in \mathcal{D}_D(x) : L(1,\chi_d)\geq e^\gamma \tau \rbrace. \vert
\end{equation}
On the other hand, we see that for $d \geq \sqrt{x}$ we have
\[
	\tau_0(d)\leq \left( 1+\frac{4\log(q+P(1))}{\log x }\right) \cdot \tau,
\]
so that
\begin{eqnarray}
	\vert S(x) \vert 
	& \geq &
	\vert \lbrace d\in \mathcal{D}_D(x) : L(1,\chi_d)\geq e^\gamma \tau_0(d), d \geq \sqrt{x} \rbrace \vert
	\nonumber
	\\
	& \geq &
	\vert \lbrace d\in \mathcal{D}_D(x) : L(1,\chi_d)\geq e^\gamma \left( 1+\frac{4\log(q+P(1))}{\log x }\right) \cdot \tau \rbrace \vert
	- \vert \mathcal{D}_D(\sqrt{x}) \vert.
	\label{eq:S(x) lower bound}
\end{eqnarray}

Using Theorems \ref{thm:Distribution} and \ref{thm:ExponentialDecay} on (\ref{eq:S(x) upper bound}) and (\ref{eq:S(x) lower bound}) and taking $x$ large, the result follows.
\end{proof}

\section{The average incidence of a class number in a family: Proof of Theorem \ref{thm:F-average}}

We closely follow the proof of Theorem 1.6 in  \cite{DL}. Recall that given a polynomial $D(n)$ of continued fraction discriminant type, $\F$ is  the number of discriminants in the family $\mathcal{D}_D$ with class number $h$. In order to obtain an asymptotic formula for $\sum_{h\leq H} \F$,
we first show that we can restrict our attention to discriminants $d\in \mathcal{D}_D$ such that  $d\leq X:=H^2(\log H)^8$.  Indeed, if $d\geq X$ and $h(d)\leq H$ then by the class number formula \eqref{eq:dirichlet-class-number-formula} we must have $L(1,\chi_d)\ll 1/(\log H)^{3}$. However, it follows from Tatuzawa's refinement of Siegel's Theorem \cite{Ta} that for large $d$, we have $L(1,\chi_d)\geq 1/(\log d)^2$ with at most one exception. Therefore, we obtain
  \begin{equation}
\label{eq:F-average-truncation}
	\sum_{h\leq H} \F
	= \sum_{\substack{d \in \mathcal{D}_D(X)\\ h(d) \leq H}} 1
	+ O(1).
\end{equation}

We estimate the main term in \eqref{eq:F-average-truncation} by using the smoothing function
\[
	I_{c,\lambda,N}(u):=\frac{1}{2\pi i}\int_{c-i\infty}^{c+i\infty}u^{s}\left(\frac{e^{\lambda s}-1}{\lambda s}\right)^{N}\frac{ds}{s},
\]
where $c=1/\log H$, $N$ is a positive integer, and $0<\lambda\leq 1$ is a real number to be chosen later. Using (\ref{eq:F-average-truncation}) together with  (4.19) of \cite{DL}, we obtain
\begin{equation}
\label{eq:F(h)-integral-inequality}
\sum_{h\leq H}\F \leq\frac{1}{2\pi i}\int_{c-i\infty}^{c+i\infty}\sum_{\substack{d\in \mathcal{D}_D(X)}
}\frac{H^{s}}{h(d)^{s}}\left(\frac{e^{\lambda s}-1}{\lambda s}\right)^{N}\frac{ds}{s}+O\left(1\right)\leq\sum_{h\leq e^{\lambda N}H}\F.
\end{equation}
By Theorem \ref{thm:ComplexMoments} and Proposition \ref{prop:average of character-ex}, there exists a constant $B>0$ such that for all $x\geq \sqrt{X}$ and any complex number $\vert z \vert \leq T := B \log X / (\log_2 X \log_3 X) $, we have
\begin{equation}\label{RecallComplex}
\sumst_{d\in\mathcal{D}_D(x)}L(1,\chi_{d})^{z}=C_{1}y(x)\mathbb{E}(L(1,\mathbb{X}_D)^{z})+O\left(x^{1/2}\exp\left(-\frac{\log x}{20\log\log x}\right)\right)
\end{equation}
where
\[
	C_1=\left( 1-\frac{c(4)}{4} \right) \prod_{p>2}\left(1-\frac{c(p)}{p^{2}}\right).
\]
Now if $\Re(z) > -1/2$, then the contribution of the possible  exceptional discriminants to the complex moments in \eqref{RecallComplex} is $\ll x^{1/4}(\log x)^{1/2}$, since there are $\ll \log x$ of them for $d \leq x$, and $L(1,\chi_d)\gg (\log d) / \sqrt{d} $. Thus, for $x\geq \sqrt{X}$ and any complex number $z$ such that $\Re(z) > -1/2$ and $|z|\leq T$, we have
\begin{equation}\label{WithExceptional}
\sum_{d\in\D}L(1,\chi_{d})^{z}=C_{1}y(x)\mathbb{E}(L(1,\mathbb{X}_D)^{z})+O\left(x^{1/2}\exp\left(-\frac{\log x}{20\log\log x}\right)\right).
\end{equation}
For brevity, we define
$$
	\ell(x):=\frac{\sqrt{x}}{\log(p(x)+q\sqrt{D(x)})}.
$$
Then we have $h(d)=\ell(d)L(1, \chi_d)$ by the class number formula \eqref{eq:dirichlet-class-number-formula}. Hence, using integration by parts, we deduce from \eqref{WithExceptional} that
\begin{multline}
\label{eq:average-h(d)^(-s)}
	\sum_{\substack{d\in\mathcal{D}_{\text{ch}}(X)}}
	h(d)^{-s}=C_1 \mathbb{E}(L(1,\mathbb{X}_D)^{-s})
	\left(
		\int_{1}^{X}y'(x)\ell(x)^{-s}		dx
	\right)
	\\
	+O\left(X^{1/2}\exp\left(-\frac{\log X}{50\log\log X}\right)\right)
\end{multline}
for $\vert s \vert \leq T $ and $\text{Re}(s) = c$.

 Since $h(d)\geq 1$ and $\vert e^{\lambda s}-1\vert\leq 3$ for large enough $H$, we see that the contribution of the region $\vert s\vert>T$ to the integral in (\ref{eq:F(h)-integral-inequality}) is 
\[
\ll X^{1/2}\left(\frac{3}{\lambda}\right)^{N}\int_{\substack{\vert s\vert>T\\
\Re(s)=c
}
}\frac{\vert ds\vert}{\vert s\vert^{N+1}}\ll\frac{X^{1/2}}{N}\left(\frac{3}{\lambda T}\right)^{N}.
\]
We also have $\vert(e^{\lambda s}-1)/\lambda s\vert\leq 4$ for large enough $H$. Therefore, it follows from (\ref{eq:average-h(d)^(-s)}) that the integral in (\ref{eq:F(h)-integral-inequality}) equals
\begin{equation}
\label{eq:F-truncated-integral}
\frac{1}{2\pi i}\int_{\substack{\vert s\vert\leq T\\
\Re(s)=c
}
}
C_1 \mathbb{E}(L(1,\mathbb{X}_D)^{-s})\left(\int_{1}^{X}y'(x)\ell(x)^{-s}	dx\right)H^{s}\left(\frac{e^{\lambda s}-1}{\lambda s}\right)^{N}\frac{ds}{s}+\mathcal{E}
\end{equation}
where
\[
	\mathcal{E}\ll\frac{X^{1/2}}{N}\left(\frac{3}{\lambda T}\right)^{N}+\frac{4^{N}T}{c}X^{1/2}\exp\left(-\frac{\log X}{50\log\log X}\right).
\]
Choosing $\lambda=e^{10}/T$ and $N=[A\log\log H]$ for a constant $A>1$ gives
\[
	\mathcal{E}\ll_{A}\frac{H}{(\log H)^{A}}.
\]
Extending the main term of (\ref{eq:F-truncated-integral}) to the entire line $\Re(s)=c$, we see that it equals
\begin{multline}
\label{eq:F-expected-integral}
\frac{1}{2\pi i}\int_{c-i\infty}^{c+i\infty}C_1\mathbb{E}(L(1,\mathbb{X}_D)^{-s})\left(\int_{1}^{X}y'(x)\ell(x)^{-s}dx\right)H^{s}\left(\frac{e^{\lambda s}-1}{\lambda s}\right)^{N}\frac{ds}{s}\\+O\left(\mathbb{E}\big(L(1,\mathbb{X}_D)^{-c}\big)\frac{X^{1/2}}{N}\left(\frac{3}{\lambda T}\right)^{N}\right)
\\
	= C_1 \mathbb{E} \left(
		\int_1^X I_{c,\lambda,N} \left(
			\frac{H}{\ell(x) L(1,\mathbb{X}_D)}\right)
		y'(x) dx
	\right)
	+ O_A\left( \frac{H}{(\log H)^A} \right).	
\end{multline}
To shorten our notation we define 
$Y = H L(1,\mathbb{X}_D)^{-1}$. Then
it follows from (4.19) of \cite{DL} that for $1< x\leq X$ we have
\[
I_{c,\lambda,N}\left(\frac{H}{\ell(x) L(1,\mathbb{X}_D)}\right)=\begin{cases}
1 & \text{ if }\ell(x)\leq Y,\\
\in[0,1] & \text{ if }Y< \ell(x)\leq e^{\lambda N}Y,\\
0 & \text{ if }\ell(x)>e^{\lambda N}Y.
\end{cases}
\]
Furthermore, note that  $
\ell(x)=(2\sqrt{x})/(\log x+\psi(x))
$ for some $\psi(x)$ that verifies $k_1 \leq \psi(x)\leq  k_2$, where $k_1=2 \log q$ and $k_2=2 \log(q+p(1)/\sqrt{D(1)})$. Thus, if for a constant $c$ we define
\[
	\ell_c(x)=\frac{2\sqrt{x}}{\log x + c},
\]
then we have $ \ell_{k_2}(x) \leq \ell(x) \leq \ell_{k_1}(x)$, and therefore
\[
I_{c,\lambda,N}\left(\frac{H}{\ell(x) L(1,\mathbb{X}_D)}\right)=\begin{cases}
1 & \text{ if }\ell_{k_1}(x)\leq Y,\\
0 & \text{ if }\ell_{k_2}(x)>e^{\lambda N}Y,\\
\in[0,1] & \text{otherwise}.\\

\end{cases}
\]
For any $c>0$ the function $\ell_c(x)$ is strictly increasing on $(e^2, \infty)$ and hence is invertible on this domain.  Let $g_c$ be its inverse function.   Then, we obtain
\begin{multline}
\label{eq:integral-I-dx}
\int_{1}^{X}I_{c,\lambda,N}\left(\frac{H}{\ell(x) L(1,\mathbb{X}_D)}\right)y'(x)dx=\min\left(y(g_{k_1}(Y)),y(X)\right)
\\
+O\left(y(g_{k_2}(e^{\lambda N}Y))-y(g_{k_1}(Y))+1\right).
\end{multline}
Note that for any $c>0$ we have 
	$g_c(x)=x^{2}\big(\log x+ O_c(\log\log x)\big)^2$ for $x\geq e^2$. Moreover, if $g_{k_1}(Y)> X$ then $Y> \ell_{k_1}(X)$ and hence $L(1,\X)\ll 1/(\log H)^3 $. 
Therefore, it follows from Theorem \ref{thm:ExponentialDecay} that
\begin{align*}
\mathbb{E}\left(\min\left(y(g_0(Y)),y(X) \right)\right)
&=\mathbb{E}\left(y(g_0(Y))\right) +O\Big(X^{1/2} \exp\left(-\log^2 H\right)\Big)\\
&= \frac{1}{\sqrt{a}}\mathbb{E}\left(L(1, \X)^{-1}\right) H\log H  + O(H\log_2 H).
\end{align*}
Furthermore, a similar argument shows that
\begin{align*}
\ex\left(y(g_2(e^{\lambda N}Y))-y(g_0(Y))\right)
&= \left(e^{\lambda N}-1\right) \mathbb{E}\left(L(1, \X)^{-1}\right) H\log H + O(H\log_2 H)\\
&\ll H (\log_2 H)^2 \log_3 H.
\end{align*}
Combining this estimate with  with equations (\ref{eq:F(h)-integral-inequality}), (\ref{eq:F-truncated-integral}), (\ref{eq:F-expected-integral}) and (\ref{eq:integral-I-dx}) we deduce
$$
\sum_{h\leq H}\F \leq \frac{C_1}{\sqrt{a}}\mathbb{E}\left(L(1, \X)^{-1}\right)H\log H + O\left(H (\log_2 H)^2 \log_3 H\right)\leq\sum_{h\leq e^{\lambda N}H}\F.
$$
Replacing $e^{\lambda N}H$ by $H$ in the right hand side inequality yields
$$
\sum_{h\leq H}\F= \frac{C_1}{\sqrt{a}}\mathbb{E}\left(L(1, \X)^{-1}\right)H\log H + O\big(H (\log_2 H)^2 \log_3 H\big).
$$
Finally, we set $C_2=C_1/\sqrt{a}$.

\subsection{Computing $C_2$}

For any $z\in \mathbb{C}$ we have
\begin{equation}
\label{eq:E(L) definition}
\ex\left(L(1,\X)^z\right)= \prod_{p} E_p(z),
\end{equation}
where
\begin{equation}
\label{eq:E_p(z) definition}
E_p(z) :=
\ex \left[ \left(\sum_{n=0}^\infty \frac{\X(p^n)}{p^n}\right)^z \right].
\end{equation}
In the case of odd $p$ we have
\begin{equation}
\label{eq:E_p(z) odd identity}
E_p(z)=\ex\left(\left(1-\frac{\X(p)}{p}\right)^{-z}\right)= \alpha_p \left(1-\frac{1}{p}\right)^{-z}+
\beta_p \left(1+\frac{1}{p}\right)^{-z}+\gamma_p.
\end{equation}

We see that the series in parentheses of (\ref{eq:E_p(z) definition}) in the case of $p=2$ is equal to
\begin{equation*}
	1+\frac{\X(2)}{2}
	+\sum_{n=0}^\infty \frac{\X(2^{2n+2})}{2^{2n+2}}
	+\sum_{n=0}^\infty \frac{\X(2^{2n+3})}{2^{2n+3}}
	= 
	1+\frac{\X(2)}{2}
	+\frac{1}{3} \left( \X(4)+\frac{\X(8)}{2} \right),
\end{equation*}
and hence by definition, we have the expression
\begin{equation}
\label{eq:E_2 identity}
E_2(z)=\sum_{\vec{x}\in \lbrace 0, \pm 1 \rbrace^3}
	\mathbb{P}\left(\begin{array}{c} \X(2)=x_1, \\ \X(4)=x_2, \\ \X(8)=x_3 \end{array}\right)
	\left(1+\frac{x_1}{2}+\frac{x_2}{3}+\frac{x_3}{6}\right)^z.
\end{equation}

We can exactly compute $C_2$ for specific families using (\ref{eq:E(L) definition}), (\ref{eq:E_p(z) odd identity}) and (\ref{eq:E_2 identity}). For example, taking Chowla's family $D(n)=4n^2+1$, we see that $C_2=1/2G$ where $G=L(1,\chi_{-1})$ is Catalan's constant. (Here, $\chi_{-1}$ is the primitive character modulo 4). For Yokoi's family, we have $D(n)=n^2+4$. However, since $D(n)$ must be squarefree, we cannot have $n$ even, and so we can without loss of generality instead take $D(n)=(2n-1)^2+4=4n^2-4n+5$. In this case, it can be shown that $C_2$ takes on the same value.
For the family arising from $(\sqrt{D(n)}+1)/2=[n+1,\overline{1,2n+1}]$, we have $D(n)=4n^2+12n+5$. Through a calculation, it can be shown that
\[
	C_2=\prod_p \left( 1-\frac{2}{p^2}\right) \left(1-\frac{1}{p^2}\right)
	=\zeta(2)^{-1}(2C_{FT}-1),
\]
where $C_{FT}$ is the Feller-Tornier constant.


\begin{thebibliography}{CKR}

\bibitem[Bi]{Bi} A. Bir\' o,
\emph{Yokoi's conjecture},
Acta Arith. \textbf{106} (2003), 85-104.

\bibitem[BG]{BG} A. Bir\' o, A. Granville,
\emph{Zeta functions for ideal classes in real quadratic fields, at $s=0$},
J. Number Theory \textbf{132} (2012), 1807-1829.

\bibitem[BL]{BL} A. Bir\' o, K. Lapkova, 
\emph{The class number one problem for the real quadratic fields $Q(\sqrt{(an)^2+4a})$},
Acta Arith. \textbf{172} (2016), 117-131. 

\bibitem[BK]{BK} V. Blomer, V. Kala, \emph{Number fields without universal $n$-ary quadratic forms}, Math. Proc. Cambridge Philos. Soc. \textbf{159} (2015), 239-252.



\bibitem[DL]{DL} A. Dahl, Y. Lamzouri, \emph{The distribution of class numbers in a special family of real quadratic fields}, submitted, arxiv:1603.00889, 27 pp. 



\bibitem[Fr]{Fr} C. Friesen, \emph{On continued fractions of given period}, Proc. Amer. Math. Soc. \textbf{103} (1988), 8-14.

\bibitem[GS] {GS} A. Granville and K. Soundararajan,
\emph{The distribution of values of $L(1, \chi_d)$.} 
Geom. Funct. Anal. \textbf{13} (2003), no. 5, 992-1028. 

\bibitem[HK]{HK} F. Halter-Koch, \emph{Continued fractions of given symmetric period},
Fibonacci Quart. \textbf{29} (1991), 298-303.



\bibitem[Ka]{Ka} V. Kala, \emph{Universal quadratic forms and elements of small norm in real quadratic fields}, Bull. Aust. Math. Soc. \textbf{94} (2016), 7-14.

\bibitem[Ka2]{Ka2} V. Kala, \emph{Norms of indecomposable integers in real quadratic fields}, J. Number Theory \textbf{166} (2016), 193-207.

\bibitem[KT]{KT} F. Kawamoto, K. Tomita,
\emph{Continued fractions and certain real quadratic fields of minimal type},
J. Math. Soc. Japan \textbf{60} (2008), 865-903. 





\bibitem[Lam1]{Lam1}  Y.  Lamzouri,
\emph{Extreme Values of Class Numbers of Real Quadratic Fields.}
Int. Math. Res. Not. IMRN (2015), no. \textbf{22}, 11847-11860.

\bibitem[Lam2]{Lam2}  Y.  Lamzouri,
\emph{The distribution of Euler-Kronecker constants of quadratic fields.}
Math. Anal. App. (2015), no. \textbf{432}, 632-653.

\bibitem[Lap]{La} K. Lapkova,
\emph{Class number one problem for real quadratic fields of a certain type},
Acta Arith. \textbf{153} (2012), 281-298.

\bibitem[Li]{Li}  J. E. Littlewood,
\emph{On the class number of the corpus $P(\sqrt{-k})$,} 
Proc. London Math. Soc. \textbf{27} (1928), 358-372.

\bibitem[Lo]{Lo} S. Louboutin,
\emph{Continued fractions and real quadratic fields},
J. Number Theory \textbf{30} (1988), 167-176.


\bibitem[McL]{McL} J. McLaughlin,
\emph{Polynomial solutions of Pell's equation and fundamental units in real quadratic fields},
J. London Math. Soc. (2) \textbf{67} (2003), 16-28. 

\bibitem[Mo]{Mo} R. A. Mollin,
\emph{A survey of class numbers of quadratic fields in relation to integer solutions of Diophantine equations}, XVI. Steierm\" arkisches mathematisches Symposium (Stift Rein/Graz, 1986), 37-48.



\bibitem[Pe]{Pe} O. Perron, \emph{Die Lehre von den Kettenbr\" uchen}, Band 1, B. G. Teubner, Stuttgart, 1954.




\bibitem[S1]{S1} A. Schinzel,
\emph{On some problems of the arithmetical theory of continued fractions},
Acta Arith. \textbf{6} 1960/1961, 393-413. 

\bibitem[S2]{S2} A. Schinzel,
\emph{On some problems of the arithmetical theory of continued fractions II},
Acta Arith. \textbf{7} 1961/1962, 287-298. 

\bibitem[Si]{Si} J. H. Silverman, \emph{The arithmetic of elliptic curves}, GTM \textbf{106} (1986).

\bibitem[Ta] {Ta} T. Tatuzawa,
\emph{On a theorem of Siegel }
Jap. J. Math. \textbf{21} (1951), 163-178. 


\bibitem[vPW]{vPW}  A. J. van der Poorten, H. C. Williams,
\emph{On certain continued fraction expansions of fixed period length},
Acta Arith. \textbf{89} (1999), 23-35. 

\bibitem[Wa]{Wa} M. Watkins, 
\emph{Class numbers of imaginary quadratic fields,} 
Math. Comp. \textbf{73} (2004), no. 246, 907-938. 

\end{thebibliography}
\end{document}